\documentclass[12pt,a4paper]{article}

\input{mathdefs.sty}

\renewcommand{\myboldfont}{\mathbb}

\usepackage[dvips, left=2.4cm, right=2.4cm , top=2.4cm, bottom=3.4cm]{geometry}

\usepackage{bm}
\usepackage{enumerate}

\newtheorem{theorem}{Theorem}
\newtheorem{lemma}[theorem]{Lemma}

\newtheorem{prop}[theorem]{Proposition}
\newtheorem{cor}[theorem]{Corollary}

\newcounter{theremark}
\newenvironment{remark}{\medskip\parindent=0pt\textbf{Remark \arabic{theremark}\addtocounter{theremark}{1}.\ }\ignorespaces}{\medskip\par\normalsize\parindent=18pt}

\numberwithin{equation}{section}
\newcommand*\samethanks[1][\value{footnote}]{\footnotemark[#1]}

\usepackage{graphicx}

\usepackage[font=small,labelfont=bf]{caption}

\begin{document}

\title{The distribution of directions in an affine lattice:\\ two-point correlations and mixed moments\thanks{The research leading to these results has received funding from the European Research Council under the European Union's Seventh Framework Programme (FP/2007-2013) / ERC Grant Agreement n. 291147. J.M. is also supported by a Royal Society Wolfson Research Merit Award.}}
\author{Daniel El-Baz\thanks{School of Mathematics, University of Bristol, Bristol BS8 1TW, U.K.} \and Jens Marklof\samethanks \and Ilya Vinogradov\samethanks}
\date{\today}

\maketitle

\begin{abstract}
We consider an affine Euclidean lattice and record the directions of all lattice vectors of length at most $T$. Str\"ombergsson and the second author proved in [Annals of Math.~173 (2010), 1949--2033] that the distribution of gaps between the lattice directions has a limit as $T$ tends to infinity. For a typical affine lattice, the limiting gap distribution is universal and has a heavy tail; it differs markedly from the gap distribution observed in a Poisson process, which is exponential. The present study shows that the limiting two-point correlation function of the projected lattice points exists and is Poissonian. This answers a recent question by Boca, Popa and Zaharescu [arXiv:1302.5067]. The existence of the limit is subject to a certain Diophantine condition. We also establish the convergence of more general mixed moments.
\end{abstract}

\section{Introduction\label{sec:intro}}

It is an interesting problem to understand the ``randomness'' in a given deterministic sequence of real numbers. Take for instance the values of a fixed binary positive quadratic form at integer lattice points. If the form is generic, i.e.~badly approximable by rational forms, numerical experiments suggest that the fine-scale statistics are the same as those of a Poisson point process. The only result to-date in this direction is the proof of the convergence of the two-point correlation function  \cite{sarnak_values_1997, Eskin05quadraticforms}, cf.~also \cite{marklof_pair_correlation_2003,marklof_pair_correlation_II_2002,margulis_quantitative_2011} for the case of inhomogeneous quadratic forms. The convergence of higher-order correlation functions has only been established in the case of generic (in measure) positive definite quadratic forms in many variables \cite{vanderkam_values_1999,vanderkam_pair_correlation_1999,VanderKam00}. The situation is similar in the problem of fine-scale 
statistics for the fractional parts of the sequence $n^2\alpha$, $n=1,\ldots,N\to\infty$, where we expect the local 
statistics to converge to those of a Poisson point process (after rescaling the sequence by a factor $N$), provided $\alpha$ is badly approximable by rationals. As in the case of binary quadratic forms, we so far only have results for the two-point correlation function \cite{rudnick_pair_1998,marklof_equidistribution_2003,heath-brown_pair_2010}. (See however \cite{rudnick_sarnak_zaharescu_2001} for the convergence of the gap distribution along special subsequences of $N$ for well approximable $\alpha$.)

In the present paper we construct a deterministic sequence whose two-point correlation function converges to the Poisson limit, although the limiting process is not Poisson. This sequence is given by the directions of vectors in an affine Euclidean lattice of length less than $T$, as $T\to\infty$.

Let $\scrL\subset\R^2$ be a Euclidean lattice of covolume one. We may write $\scrL=\Z^2 M_0$ for a suitable $M_0\in\SL(2,\R)$. For $\vecxi=(\xi_1,\xi_2)\in\R^2$, we define the associated affine lattice as $\scrL_\vecxi=(\Z^2+\vecxi)M_0$. Denote by $\scrP_T$ the set of points $\vecy\in\scrL_\vecxi\setminus\{\vecnull\}$ inside the open disc of radius $T$ centered at zero or, more generally, in the annulus $cT< \|\vecy\| <T$ for some fixed $c\in[0,1)$. The number $N(T)$ of points in $\scrP_T$ is asymptotically
\begin{equation}
N(T)  \sim \pi (1-c^2) T^2 , \qquad T\to\infty.
\end{equation}
We are interested in the distribution of directions $\|\vecy\|^{-1} \vecy$ as $\vecy$ ranges over $\scrP_T$, counted {\em with} multiplicity. That is, if there are $k$ lattice points corresponding to the same direction, we will record that direction $k$ times. For each $T$, this produces a finite sequence of $N(T)$ unit vectors $(\cos(2\pi\alpha_j),\sin(2\pi\alpha_j))$ with $\alpha_j=\alpha_j(T)\in\T=\R/\Z$ and $j=1,\ldots,N(T)$. It is well known that the set of directions is uniformly distributed as $T\to\infty$: for any interval $U\subset\T$ we have
\begin{equation}\label{udi}
\lim_{T\to\infty} \frac{\# \{ j \le N(T) \colon  \alpha_j \in U \}}
{N(T)} 
= |U| ,
\end{equation}
where $|\cdot|$ denotes length. 
Given a bounded interval $I \subset \R$, define the subinterval $J=J_N(I,\alpha)=N^{-1} I +\alpha+\Z\subset\T$ of length $N^{-1}|I|$, and ask for the number of directions $\alpha_j$ that fall into this small interval:
\begin{equation}
\scrN_{c,T}(I,\alpha) = \#\{ j\le N(T) \colon \alpha_j\in J_{N(T)}(I,\alpha) \}.
\end{equation}
With this choice, \eqref{udi} implies that for any Borel probability measure $\lambda$ on $\T$ with continuous density,
\begin{equation}\label{expectation}
\lim_{T\to\infty} \int_\T \scrN_{c,T}(I,\alpha)\, \lambda(d\alpha) = | I |.
\end{equation}
It is proved in \cite{marklof_strombergsson_free_path_length_2010} that for every $\vecxi\in\R^2$ and $\alpha\in\T$ random with respect to $\lambda$ (which is only assumed to be absolutely continuous with respect to Lebesgue measure), the random variable $\scrN_{c,T}(I,\alpha)$ has a limit distribution $E_{c,\vecxi}(k,I)$. That is, for every $k\in\Z_{\ge 0}$,
\begin{equation}\label{one}
\lim_{T\to\infty}\lambda(\{ \alpha\in\T \colon \scrN_{c,T}(I,\alpha) = k\}) = E_{c,\vecxi}(k,I).
\end{equation}
The limit distribution $E_{c,\vecxi}(k,I)$ is independent of the choice of $\lambda$, $\scrL$ and, if $\vecxi\notin\Q^2$, independent of $\vecxi$. In fact, these results hold for several test intervals $I_1,\ldots, I_m$, and follow directly from Theorem 6.3, Remark 6.4  and Lemma 9.5 of \cite{marklof_strombergsson_free_path_length_2010} 
for $\vecxi\notin\Q^2$ and from Theorem 6.5, Remark 6.6 and Lemma 9.5 of \cite{marklof_strombergsson_free_path_length_2010}  in the case $\vecxi\in\Q^2$:

\begin{theorem}\label{th:prelim}
Fix $\vecxi\in\R^2$ and let $I=I_1\times\cdots \times I_m\subset \R^m$ be a bounded box. Then there is a probability distribution $E_{c,\vecxi}(\,\cdot\,,I)$ on $\Z_{\ge 0}^m$ such that, for any $\veck=(k_1,\ldots,k_m)\in\Z_{\ge 0}^m$ and any Borel probability measure $\lambda$ on $\T$, absolutely continuous with respect to Lebesgue,
\begin{equation}
\lim_{T\to\infty}\lambda(\{ \alpha\in\T \colon \scrN_{c,T}(I_1,\alpha) = k_1,\ldots,\scrN_{c,T}(I_m,\alpha) = k_m\}) = E_{c,\vecxi}(\veck,I).
\end{equation}
\end{theorem}

In the language of point processes, Theorem \ref{th:prelim} says that the point process \[\{ N(T) (\alpha_j-\alpha+\Z)\}_{j\le N(T)}\] on the torus $\R/(N(T)\Z)$ converges, as $T\to\infty$, to a random point process on $\R$ which is determined by the probabilities $E_{c,\vecxi}(\veck,I)$.
We will give a precise characterization of $E_{c,\vecxi}(\veck,I)$ in Section \ref{sec:props}, and now only highlight the following key properties:
\begin{enumerate}[(a)]
\item $E_{c,\vecxi}(\veck,I)$ is independent of $\lambda$ and $\scrL$.
\item $E_{c,\vecxi}(\veck,I+r\underline{e})=E_{c,\vecxi}(\veck,I)$ for any $r\in\R$, where $\underline{e}=(1,1,\ldots,1)$; that is, the limiting process is translation invariant.
\item $\sum_{\veck\in\Z_{\ge 0}^m} k_j E_{c,\vecxi}(\veck,I) = \sum_{k=0}^\infty k E_{c,\vecxi}(k,I_j) = |I_j|$ for any $j\le m$.
\item For $\vecxi\in\Q^2$, $\sum_{\veck\in\Z_{\ge 0}^m} \|\veck\|^\sigma E_{c,\vecxi}(\veck,I) < \infty$ for $0\le \sigma<2$, and $=\infty$ for $\sigma\ge 2$. 
\item For $\vecxi\notin\Q^2$, $E_{c,\vecxi}(\veck,I)=:E_{c}(\veck,I)$ is independent of $\vecxi$.
\item $\sum_{\veck\in\Z_{\ge 0}^m} \|\veck\|^\sigma E_{c}(\veck,I) < \infty$ for $0\le \sigma<3$, and $=\infty$ for $\sigma\ge 3$. 
\end{enumerate}

Theorem \ref{th:prelim} implies for example that the distribution of spacings between each element and its $k$th neighbor to the right has a limit distribution, cf.~Figure \ref{fig:gaps}.

\begin{figure}[t]
\centerline{\includegraphics[width=0.9\textwidth]{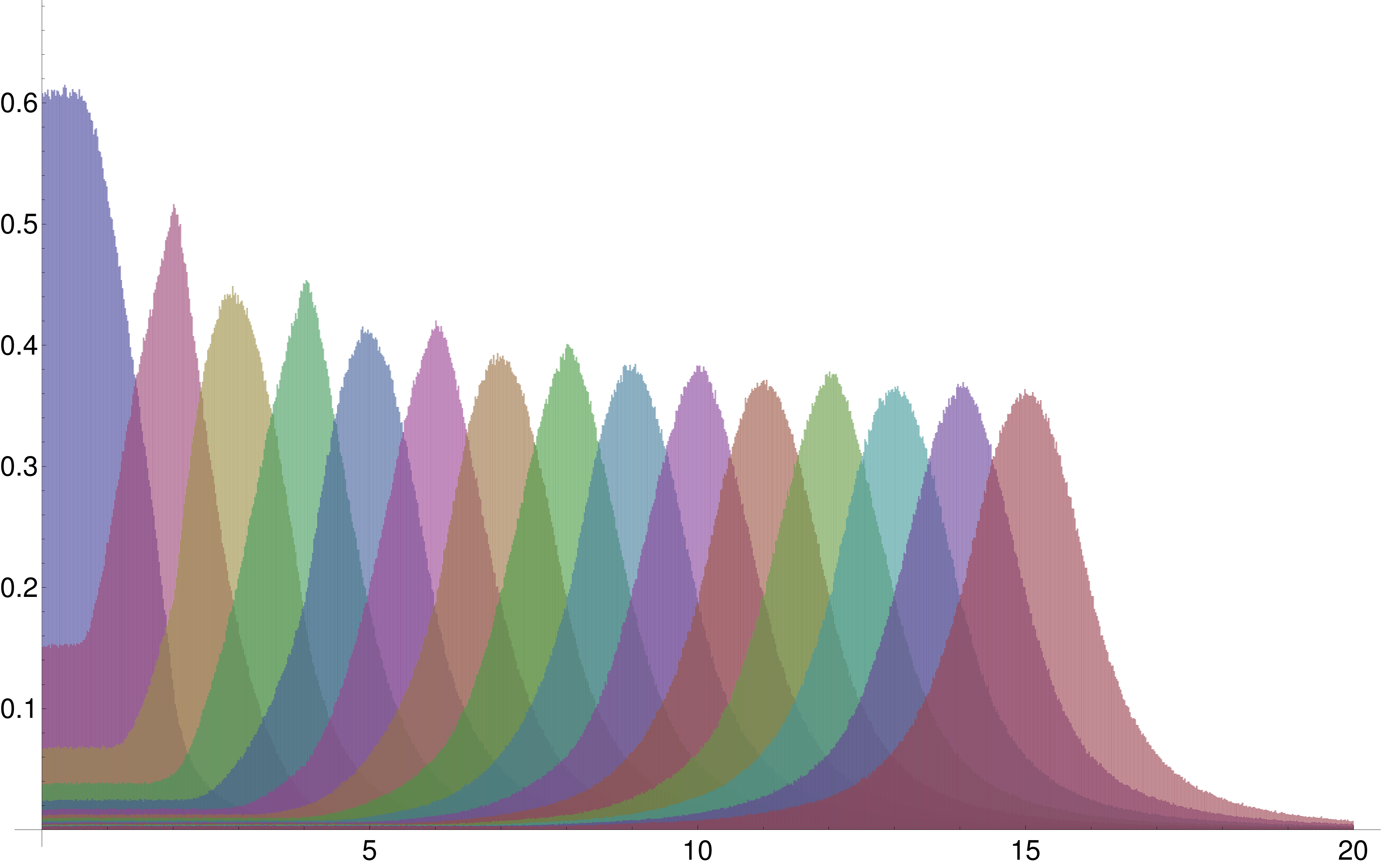}}
\caption{The figure shows the distribution of spacings between each element $\alpha_j$ and its $k$th neighbor to the right, for $k=1,\ldots,15$ and $\vecxi=(\sqrt[3]4,\sqrt[3]2)$, $T=1000$. The case $k=1$ corresponds to the gap distribution.  }
\label{fig:gaps}
\end{figure}

Properties (d) and (f) imply that the limiting process is {\em not} a Poisson process. We will however see that the second moments and two-point correlation functions are those of a Poisson process with intensity $1$, when $\vecxi\notin\Q^2$. Specifically, we have
\begin{equation}\label{Ec2}
\sum_{\veck\in\Z_{\ge 0}^2} k_1 k_2 E_{c}(\veck,I_1\times I_2) = |I_1\cap I_2|+|I_1|\,|I_2|
\end{equation}
and, in particular, 
\begin{equation}\label{Ec1}
\sum_{k=0}^\infty k^2 E_{c}(k,I_1) = |I_1|+|I_1|^2 ,
\end{equation}
which coincide with the corresponding formulas for the Poisson distribution. 

The main result of the study presented here is to establish the convergence to the finite moments of the limiting process. It is interesting that the convergence of certain moments requires a Diophantine condition on $\vecxi$. We say that  {\em $\vecxi \in \R^2$ is Diophantine of type $\kappa$} if there exists $C>0$  such that 
\beq\forall \bm r = (r_1, r_2) \in \Z^2 \setminus \{\vecnull\}, \forall m \in \Z, \, |\bm r\cdot \vecxi + m| \ge \frac C{(|r_1|+|r_2|)^{\kappa}}.\eeq
It is well known that Lebesgue almost all $\vecxi\in\R^2$ are Diophantine of type $\kappa>2$, and that there is no $\vecxi\in\R^2$ which is Diophantine of type $\kappa<2$ \cite{schmidt_diophantone_1980}.
A specific example of a Diophantine vector of type $\kappa=2$ can be obtained from a degree 3 extension $K$ over $\Q$: If $\xi_1, \xi_2\in K$ are such that $\{1,\xi_1,\xi_2\}$ is a $\Q$-basis for $K$, then $\vecxi=(\xi_1,\xi_2)$ is Diophantine of type $2$ (see Theorem III of Chapter 5 and its proof in \cite{Cassels57DiophantineApproximation}). 

The appearance of Diophantine conditions for the convergence of moments is reminiscent of the same phenomenon in the quantitative Oppenheim conjecture, in particular the pair correlation problem for the values of quadratic forms at integers \cite{Eskin05quadraticforms, marklof_pair_correlation_II_2002, marklof_pair_correlation_2003}. The techniques we use here generalize the approach in \cite{marklof_pair_correlation_II_2002, marklof_mean_square_2005}. 

For $I=I_1\times\cdots\times I_m\subset  \R^m$, $\lambda$ a Borel probability measure on $\T$ and $\vecs=(s_1,\ldots,s_m)\in\C^m$ let
\begin{equation}
\M_\lambda(T,\vecs):=\int_\T (\scrN_{c,T}(I_1,\alpha)+1)^{s_1} \cdots (\scrN_{c,T}(I_m,\alpha)+1)^{s_m}  \lambda(d\alpha).
\end{equation}

We denote the positive real part of $z\in\C$ by $\re_+(z):=\max\{ \re(z), 0 \}$.

\begin{theorem}\label{th:main0}
Let $I=I_1\times\cdots\times I_m\subset  \R^m$ be a bounded box, and $\lambda$ a Borel probability measure on $\T$ with continuous density. Choose $\vecxi\in\R^2$ and $\vecs=(s_1,\ldots,s_m)\in\C^m$, such that one of the following hypotheses holds:
\begin{enumerate}[{\rm ({A}1)}]
\item  $\re_+(s_1)+\ldots+\re_+(s_m)< 2$.
\item  $\vecxi$ is Diophantine of type $\kappa$, and $\re_+(s_1)+\ldots+\re_+(s_m)<2+\frac2{\kappa}$. 
\end{enumerate}
Then
\begin{equation}\label{limit:main}
\lim_{T\to\infty} \M_\lambda(T,\vecs) = \sum_{\veck\in\Z_{\ge 0}^m} (k_1+1)^{s_1}\cdots (k_m+1)^{s_m} E_{c,\vecxi}(\veck,I).
\end{equation}
\end{theorem}

\begin{remark}\label{rem:div}
The fact that some Diophantine condition is necessary in (A2) can be seen from the following argument. Assume that $\vecr\cdot(\vecxi+\vecm) = 0$ for some $\vecr\in\Z^2\setminus\{\vecnull\}$, $\vecm\in\Z^2$. Then there is a line through the origin (in direction $\alpha_\vecr$, say) that contains infinitely many lattice points of $\scrL_\vecxi$ so that, for any $\epsilon>0$ and sufficiently large $T$,
\begin{equation}
\scrN_{c,T}((-\epsilon,\epsilon),\alpha_\vecr) \gg_{\vecr,\scrL} (1-c)T,
\end{equation}
where the implied constant depends only on $\vecr$ and $\scrL$. This in turn implies that when $\lambda $ is the Lebesgue measure and $s\ge 2$ we have
\begin{equation}
\M_1(T,s) \gg_{c,\vecr,\scrL}  T^{s-2}, 
\end{equation}
and thus any moment with $s>2$ diverges. In the case $s=2$ we have for any bounded interval $I_1\subset\R$
\begin{equation}\label{limit:main111}
\liminf_{T\to\infty} \M_1(T,2) > \sum_{k\in\Z_{\ge 0}^m} (k+1)^{2} E_{c,\vecxi}(\veck,I_1).
\end{equation}
The Diophantine condition in (A2) is however by no means sharp. The statement of Theorem \ref{th:main0} remains valid if in (A2) we use vectors of the form $\vecxi=\vecn \omega +\vecl$ where $\vecn\in\Z^2 \setminus \{ \vecnull \}$ and $\vecl\in\Q^2$ so that $\det(\vecn,\vecl)\notin\Z$, and $\omega\in\R$ is Diophantine of type $\frac{\kappa}{2}$, i.e.~there exists $C>0$ such that $|r\omega+m|\ge C |r|^{-\kappa/2}$ for all $r\in\Z\setminus\{0\}$, $m\in\Z$; we still require that $\re_+(s_1)+\ldots+\re_+(s_m)<2+\frac2{\kappa}$. The proof of this claim follows the same argument as 
the one used for the
two-dimensional Diophantine condition, see Section \ref{sec:singular} for details. Note that Lebesgue almost all $\omega$ are of type $\frac{\kappa}{2}>1$, and type $1$ is the smallest possible, achieved for instance by quadratic surds.
\end{remark}

To explain the key step in the proof of Theorem \ref{th:main0}, define the restricted moments
\begin{equation}
\M_\lambda^{(K)}(T,\vecs):=\int_{\max_j \scrN_{c,T}(I_j,\alpha)\le K} (\scrN_{c,T}(I_1,\alpha)+1)^{s_1} \cdots (\scrN_{c,T}(I_m,\alpha)+1)^{s_m}  \lambda(d\alpha) .
\end{equation}
Theorem \ref{th:prelim} now implies that, for any $K\ge 0$,
\begin{equation}\label{limit:main2}
\lim_{T\to\infty} \M_\lambda^{(K)}(T,\vecs) = \sum_{\substack{\veck\in\Z_{\ge 0}^m\\ |\veck|\le K}} (k_1+1)^{s_1}\cdots (k_m+1)^{s_m} E_{c,\vecxi}(\veck,I),
\end{equation}
where $|\veck|$ denotes the maximum norm of $\veck$. What thus remains to be shown in the proof of Theorem \ref{th:main0} is that under (A1) or (A2),
\begin{equation}\label{limit:main3}
\adjustlimits\lim_{K\to\infty} \limsup_{T\to\infty} \left|\M_\lambda(T,\vecs)-\M_\lambda^{(K)}(T,\vecs)\right| = 0 .
\end{equation}
We will prove this statement in Section \ref{sec:main}.

With \eqref{Ec2}, Theorem \ref{th:main0} has the following implications:

\begin{cor}\label{Cor1}
Let $I=I_1\times I_2\subset  \R^2$ and $\lambda$ be as in Theorem \ref{th:main0}, and assume $\vecxi\in\R^2$ is Diophantine. Then
\begin{equation}
\lim_{T\to\infty} \int_\T \scrN_{c,T}(I_1,\alpha) \; \scrN_{c,T}(I_2,\alpha)\; \lambda(d\alpha) = |I_1\cap I_2|+|I_1|\;|I_2| .
\end{equation}
\end{cor}

For $f\in C_0(\T^2\times\R)$ (continuous, real-valued and with compact support), we define the two-point correlation function 
\beq R_N^2(f)=\frac1N\sum_{m\in\Z}\sum_{j_1\ne j_2} f(\alpha_{j_1},\alpha_{j_2},N(\alpha_{j_1}-\alpha_{j_2}+m)).\label{eq:twopoint}\eeq
Recall that $N=N(T)$ and $\alpha_j=\alpha_j(T)$ depend on the choice of $T$. A standard argument (see Appendix \ref{324}) shows that Corollary \ref{Cor1} implies

\begin{cor}\label{cor2}
Assume $\vecxi\in\R^2$ is Diophantine. Then, for any $f\in C_0(\T^2\times\R)$
\begin{equation}\label{paircon}
\lim_{T\to\infty}  R_{N(T)}^2(f) = \int_{\T\times\R} f(\alpha,\alpha,s)\,d\alpha\, ds.
\end{equation}
\end{cor}

This answers a recent question by Boca, Popa and Zaharescu \cite{boca_pair_2013}. Figure \ref{fig:pair} shows a numerical computation of the pair correlation statistics for $\vecxi=(\sqrt[3]4,\sqrt[3]2)$, $T=1000$, which is close to the limiting density $1$ predicted by Corollary \ref{cor2}. 

\begin{figure}[t]
\centerline{\includegraphics[width=0.9\textwidth]{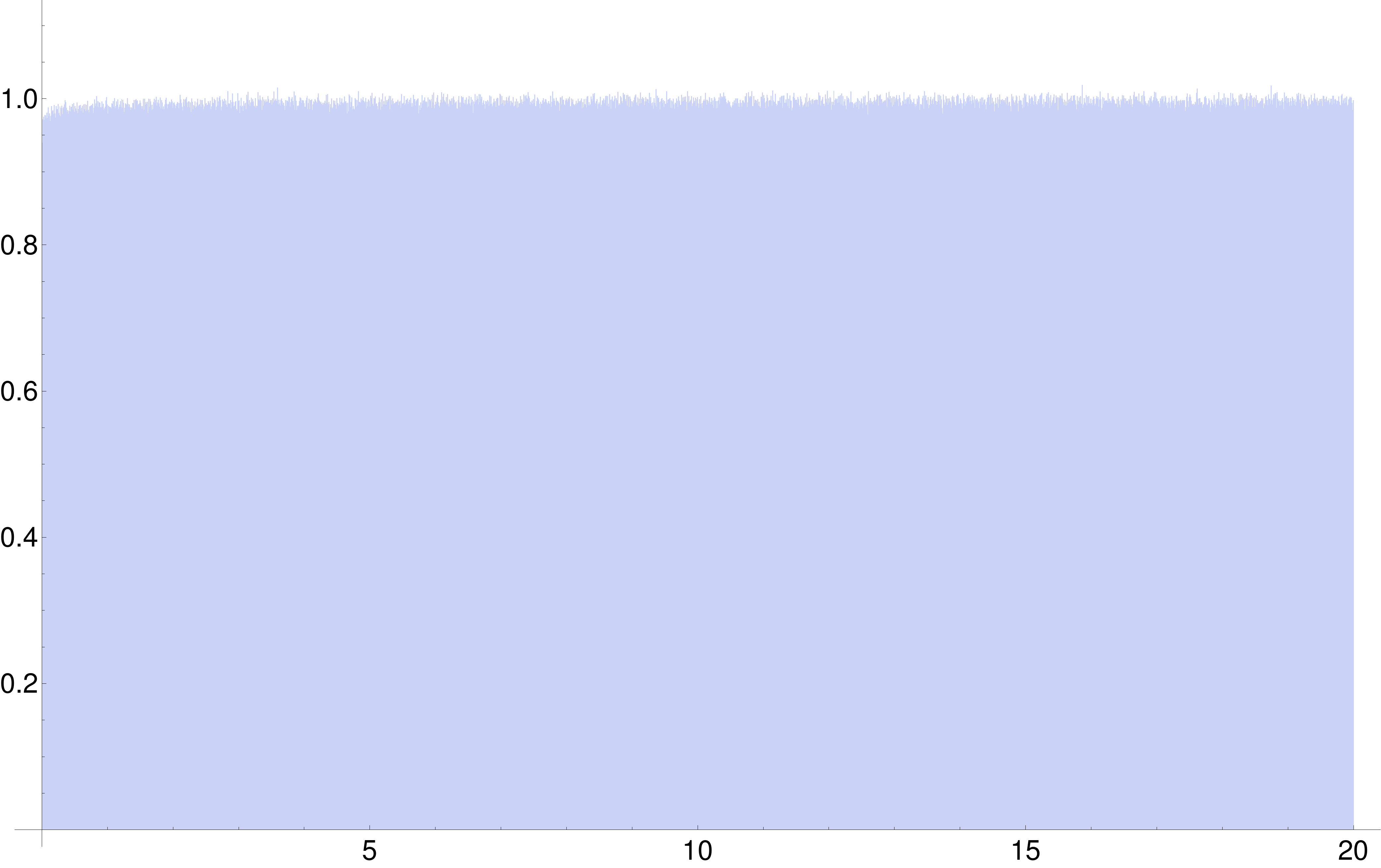}}
\caption{The figure shows a numerical computation of the pair correlation density, for $\vecxi=(\sqrt[3]4,\sqrt[3]2)$, $T=1000$. The computed density is close to $1$, as predicted by Corollary \ref{cor2}. Note that the displayed histogram can be obtained as the sum over all $k$th neighbor spacing distributions. }
\label{fig:pair}
\end{figure}

\begin{remark}
Boca and Zaharescu \cite{boca_correlations_2006} established the convergence of the pair correlation of directions in the lattice $\Z^2+\vecxi$  {\em on average} over $\vecxi$, in the case of lattice points in the square $[-T,T]^2$ (rather than a disc of radius $T$). Our approach can be adapted to this case, and to more general star shaped domains $\scrD$ dilated by $T$. Provided the projected lattice points in $T\scrD$ have a continuous limiting density $\rho_\scrD$ on $\T$, we have, under the conditions of Corollary~\ref{cor2},
\begin{equation}\label{paircon2}
\lim_{T\to\infty}  R_{N(T)}^2(f) = \int_{\T\times\R} f(\alpha,\alpha,s)\,\rho_\scrD(\alpha)^2\, d\alpha\, ds.
\end{equation}
For instance in the case when $\scrD$ is the square $[-1,1]^2$, we have 
\begin{equation}
\rho_\scrD(\alpha) = \frac{\pi}{4\cos^2[2\pi(\alpha-\nu)]} \qquad \text{if } \alpha\in[-\tfrac18,\tfrac18]+\nu,\quad \nu=0,\tfrac14,\tfrac12,\tfrac34. 
\end{equation}
In particular,
\begin{equation}
\int_{\T} \rho_\scrD(\alpha)^2\, d\alpha = \frac{\pi}{3},
\end{equation}
which yields the constant observed in \cite{boca_correlations_2006}.
The proof of \eqref{paircon2} follows from Corollary \ref{cor2} by choosing test functions $f(\alpha,\beta,s)$ whose support in $\alpha$ and $\beta$ is in an $\epsilon$-neighborhood around any given $\alpha_0,\beta_0\in\T$, with $\epsilon>0$ arbitrarily small. 
\end{remark}

\begin{remark}
The work of Elkies and McMullen \cite{ElkiesMcM04} shows that the gap distribution and other local statistics of the fractional parts of $\sqrt{n}$, $n=1,\ldots,N,$ is governed by the same limiting point process as in Theorem \ref{th:prelim}. We prove the analogue of Theorem \ref{th:main0} in this case \cite{el-baz_two-point_2013}. Note that the sparse subsequence of perfect squares leads to similar divergences as those discussed in Remark \ref{rem:div}, and should therefore be removed.
\end{remark}

\begin{remark}
In the case $\vecxi=\vecnull$, it is natural to restrict the attention to primitive lattice points, i.e., consider the set of directions without multiplicity. In this case the problem is closely related to the statistics of Farey fractions. A major difference to the present study is that in the case of primitive lattice points all moments are finite, and the analogue of Theorem \ref{th:main0} holds without any restriction on $\vecs$. The second and higher moments are non-Poissonian \cite{BocaZaharescu05CorrelationsFarey}.
\end{remark}

\begin{remark}
The characterization of point processes whose two-point statistics is Poisson was a popular problem in the statistics literature of the 1970s, see e.g.~\cite{BaddeleySilverman84_biometrics} and the literature surveyed in Section 3. Note that the process constructed by Kallenberg \cite{kallenberg_counterexample_1977} is based on the space of lattices (as pointed out by Kingman in \cite{kallenberg_counterexample_1977}) and therefore closely related to the limiting process in Theorem \ref{th:prelim}.
\end{remark}

\section{The space of (affine) lattices}

Let $G=\SL(2,\R)$ and $\Gamma=\SL(2,\Z)$. Define $G'=G\ltimes \R^2$ by 
\beq (M,\vecxi)(M',\vecxi')=(MM',\vecxi M'+\vecxi'),\eeq
and let $\Gamma'=\Gamma\ltimes \Z^2$ denote the integer points of this group. 
In the following, we will embed $G$ in $G'$ via the homomorphism $M\mapsto (M,\vecnull)$ and identify $G$ with the corresponding subgroup in $G'$.
We will refer to the homogeneous space $\Gamma\quot G$ as the space of lattices and $\Gamma'\quot G'$ as the space of affine lattices. The natural right action of $G'$ on $\R^2$ is given by  $\vecx\mapsto\vecx (M,\vecxi) := \vecx M + \vecxi$, with $(M,\vecxi)\in G'$.


Given a bounded interval $I\subset\R$ and $c\ge 0$, define the triangle/trapezoid
\begin{equation}\label{CcI}
\fC_c(I) =\{ (x,y)\in\R^2 \colon c< x<1, \; (1-c^2) y\in 2 x I   \}
\end{equation}
and set, for $g\in G'$ and any bounded subset $\fC\subset\R^2$,
\begin{equation}
\scrN(g,\fC)= \# ( \fC \cap \Z^2 g) .
\end{equation}
By construction, $\scrN(\,\cdot\,,\fC)$ is a function on the space of affine lattices, $\Gamma'\quot G'$.

Let 
\begin{equation}
\Phi^t =\begin{pmatrix} \e^{-t/2} & 0 \\ 0 & \e^{t/2} \end{pmatrix}, \qquad
k(\phi) = \begin{pmatrix} \cos\phi & -\sin\phi \\ \sin\phi & \cos\phi \end{pmatrix}.
\end{equation}

\begin{figure}
\centerline{\includegraphics[width=0.7\textwidth]{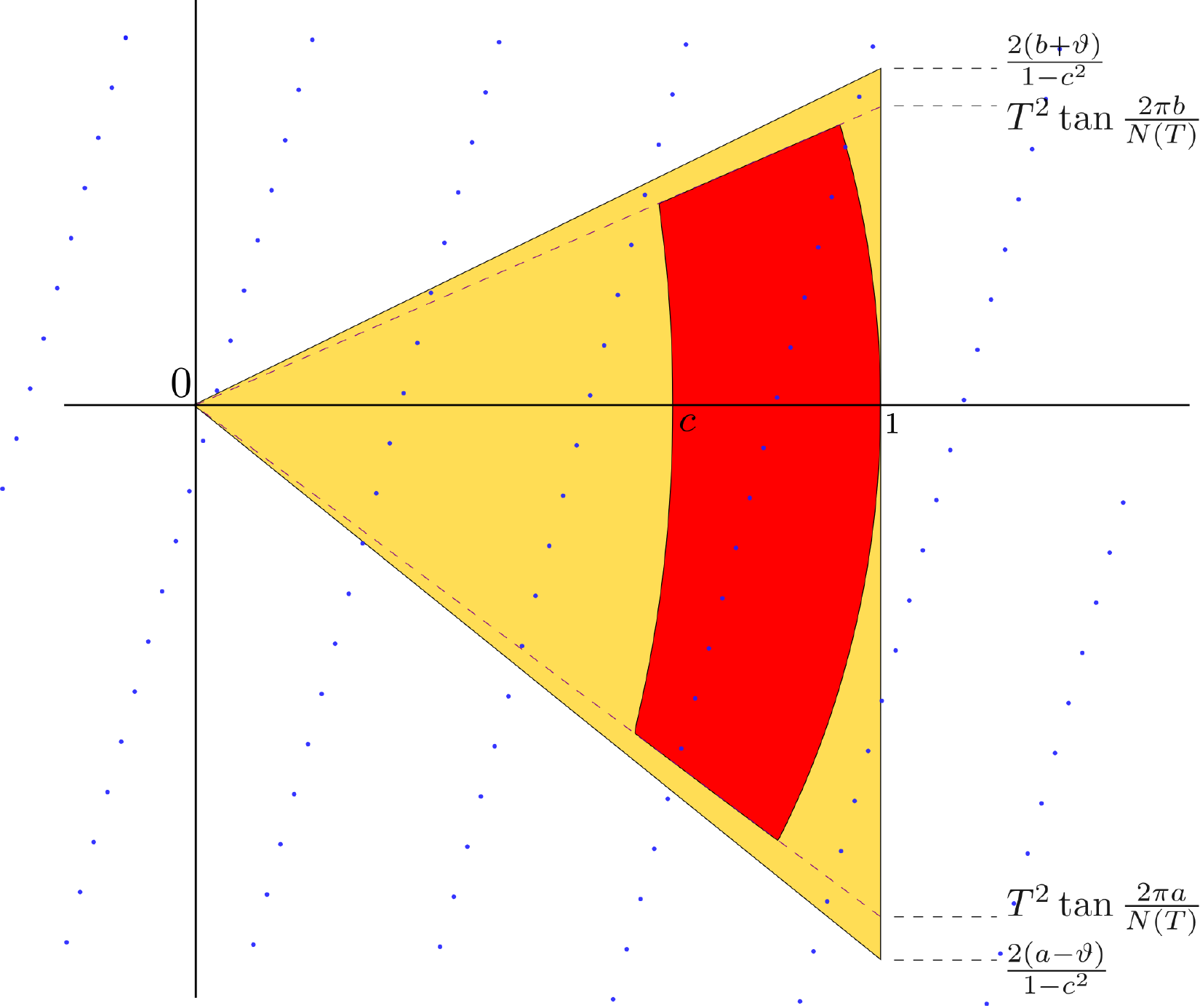}}
\caption{Here $I=[a,b]$ with $a<0<b$. The dark (red) area corresponds to counting in $\scrN_{c,T}(I,\alpha)$, while the grey (yellow) triangle is the bound we use in \eqref{crude}.}
\label{fig:triangle}
\end{figure}

An elementary geometric argument shows that, given $I\subset\R$ and $\vartheta>0$, there exists $T_0>0$ such that for all $\alpha\in\T$, $\vecxi\in\R^2$, $M_0\in \Gamma'\quot G'$ and $T=\e^{t/2}\ge T_0$,
\begin{equation}\label{crude}
\scrN_{c,T}(I,\alpha) \le \scrN\left((1,\vecxi) M_0 k(2\pi\alpha)\Phi^{t}, \fC_0(I + [-\vartheta,\vartheta]) \right).
\end{equation}
Indeed, the quantity on the left hand side counts the number of lattice points in the intersection of an annulus and a cone, while that on the right hand side counts lattice points in a triangle that properly contains the closure of this set. This is illustrated in Figure~\ref{fig:triangle}.
The observation \eqref{crude} relates our original counting function $\scrN_{c,T}(I,\alpha)$ to a function on the space of lattices. Since we will only require upper bounds, the crude estimate \eqref{crude} is sufficient. 

A more refined statement is used in \cite[Sect.~9.4]{marklof_strombergsson_free_path_length_2010}, where sets $\fC^{(t)}(I)$ are constructed such that $\scrN_{c,T}(I,\alpha) = \scrN\left((1,\vecxi) k(2\pi\alpha)\Phi^{t}, \fC^{(t)}(I) \right)$ and the sequence of sets $\fC^{(t)}(I)$ converges to $\fC_c(I)$  as $t\to\infty$. 

A convenient parametrization of $M\in G$ is given by the the Iwasawa decomposition
\begin{equation}\label{iwasawa}
M=n(u)a(v)k(\varphi)
\end{equation}
where
\begin{equation}
n(u)=\begin{pmatrix} 1 & u \\ 0 & 1 \end{pmatrix},\qquad
a(v)=\begin{pmatrix} v^{1/2} & 0 \\ 0 & v^{-1/2} \end{pmatrix} ,
\end{equation}
with $\tau=u+\mathrm{i} v$ in the complex upper half plane $\H=\{ u+\mathrm{i} v \in\C \colon v>0\}$ and $\phi\in[0,2\pi)$. 
A convenient parametrization of $g\in G'$ is then given by $\H\times [0,2\pi) \times \R^2$ via the decomposition
\begin{equation}\label{ida}
g = (1,\vecxi) n(u)a(v)k(\varphi)\eqqcolon(\tau,\phi;\vecxi).
\end{equation}
In these coordinates, left multiplication on $G$ becomes the (left) group action
\begin{equation}
g\cdot (\tau,\varphi;\vecxi)=(g\tau,\varphi_g;\vecxi g^{-1})
\end{equation}
where for 
\begin{equation}
g =(1,\vecm)\abcd
\end{equation}
we have:
\begin{equation}
g\tau = u_g+\mathrm{i} v_g = \frac{a\tau+b}{c\tau+d}
\end{equation}
and thus
\begin{equation}
v_g=\im (g \tau)=\frac{v}{|c\tau+d|^2};
\end{equation}
furthermore
\begin{equation}
\varphi_g=\varphi+\arg(c\tau+d),
\end{equation}
and
\begin{equation}
\vecxi g^{-1} = (d\xi_1-c\xi_2,-b\xi_1+a\xi_2) - \vecm  .
\end{equation}

The space of lattices has one cusp, which in the above coordinates appears at $v\to\infty$. The following lemma tells us that $\scrN(g,\fC)$ is bounded in the cusp unless $-\xi_1$ is close to an integer, in case of which the function is at most of order $v^{1/2}$.

\begin{lemma}\label{lem:upper}
For any bounded $\fC\subset\R^2$, $g=(1,\vecxi)(M,0)\in G'$ with $M$ as in \eqref{iwasawa} and $v\ge 1$,
\begin{equation}\label{eq101}
\scrN(g,\fC) \le (2 r v^{1/2}+1) \, \# ((\Z+\xi_1) \cap [-rv^{-1/2},rv^{-1/2}])
\end{equation}
where $r=\sup \{\|\vecx\| \colon \vecx\in\fC \}$. If $v> 4 r^2$ then, for any $\sigma\ge 0$,
\begin{equation}\label{eq102}
\scrN(g,\fC)^\sigma \le (2 r v^{1/2}+1)^\sigma \, \# ((\Z+\xi_1) \cap [-rv^{-1/2},rv^{-1/2}]).
\end{equation} 
\end{lemma}

\begin{proof}
Let $\fD_r$ be the smallest closed disk of radius $r$ centered at zero which contains $\fC$. Then
\begin{equation}
\begin{split}
\scrN(g,\fC) & \le \scrN(g,\fD_r) \\
& = \# ( \fD_r \cap (\Z^2+\vecxi) n(u) a(v) ) \\
& \le \# ( [-r,r]^2 \cap (\Z^2+\vecxi) n(u) a(v) ) \\
& = \# ( ([-rv^{-1/2},rv^{-1/2}]\times[-rv^{1/2},rv^{1/2}]) \cap (\Z^2+\vecxi) n(u) ) \\
& \le \sup_{\xi_2} \# ( ([-rv^{1/2},rv^{1/2}]) \cap (\Z+\xi_2) )\times \# ( [-rv^{-1/2},rv^{-1/2}] \cap (\Z+\xi_1) ) \\
& \le (2r v^{1/2}+1)\times \# ( [-rv^{-1/2},rv^{-1/2}] \cap (\Z+\xi_1) ) .
\end{split}
\end{equation}
This proves \eqref{eq101}. The second inequality \eqref{eq102} follows from the fact that $\# ((\Z+\xi_1) \cap [-rv^{-1/2},rv^{-1/2}])\in\{0,1\}$.
\end{proof}

To deal with the case of mixed moments, we note that 
\begin{equation}
\scrN(g,\fC_1)^{\sigma_1}\cdots \scrN(g,\fC_m)^{\sigma_m}\le \scrN(g,\fC_1\cup\cdots\cup\fC_m)^{\sigma_1+\ldots+\sigma_m}. 
\end{equation}
%
%
%
%

\section{Properties of the limiting distribution\label{sec:props}}

In what follows we write $f\ll g$ if there exists a positive constant $C$ such that $|f|\le C|g|$, and $f\asymp g$ means $f\ll g\ll f$.

According to \cite{marklof_strombergsson_free_path_length_2010} (see Theorems 6.3, 6.5 and subsequent remarks, and Lemma 9.5), the limit distribution of Theorem \ref{th:prelim} is given by 
\begin{equation}\label{eq:limit_measures}
E_{c,\vecxi}(\veck,I)=
	\begin{cases}
	\mu_1(\{ M\in X_1: \#( \Z^2 M \cap \fC_c(I_j))= k_j \forall j\}) & \text{if $\vecxi\in\Z^2$}\\
	\mu_q(\{ M\in X_q: \#( (\Z^2+\frac{\vecp}{q}) M \cap \fC_c(I_j))= k_j \forall j \}) & \text{if $\vecxi=\frac{\vecp}{q}\in\Q^2\setminus\Z^2 $}\\
	\mu(\{ g\in X: \#( \Z^2 g \cap \fC_c(I_j))= k_j \forall j \}) & \text{if $\vecxi\notin\Q^2$,}
	\end{cases}
\end{equation}
with $\fC_c(I_j)$ as in \eqref{CcI}. The measures $\mu$, $\mu_1$, and $\mu_q$ are the Haar probability measures on the homogeneous spaces
\begin{equation}
 X=\Gamma'\quot G',\qquad 
 X_1=\Gamma\quot G, \qquad 
 X_q=\Gamma_q\quot G,
\end{equation}
respectively. Here $\Gamma_q$ denotes the congruence subgroup
\begin{equation}
\Gamma_q=\{\gamma\in\Gamma\colon \gamma\equiv 1\bmod q\} .
\end{equation}
We have explicitly (with $M$ as in \eqref{iwasawa})
\begin{equation}\label{eq:measures}
d\mu_1(M) = \frac{3}{\pi^2} \; \frac{du\,dv\,d\phi}{v^2},\qquad d\mu(M,\vecxi) = d\mu_1(M) \, d\vecxi .
\end{equation}

The limit distribution \eqref{eq:limit_measures} is evidently independent of $\lambda$ and $\scrL$, stated as property (a) in Section \ref{sec:intro}.  

Property (b) follows from two facts. Firstly, for any $u\in\R$, we have
\begin{equation}
\fC_c(I_j) n(u) = \fC_c(I_j+u).
\end{equation}
Secondly, the measures $\mu_1$, $\mu_q$, $\mu$ are $\SL(2,\R)$-invariant, and $n(u)\in\SL(2,\R)$.

Property (e) follows from the invariance of $\mu$ under translations $\{1\}\ltimes\R^2$. 

Property (c) follows from \eqref{expectation}. It may also be derived directly from Siegel's integral formula \cite{siegel45} 
\begin{equation}\label{Siegel}
\int_{X_1} \sum_{\vecm\in\Z^2\setminus\{\vecnull\}} F(\vecm M) d\mu_1(M) = \int_{\R^2} F(\vecx)\,d\vecx ,
\end{equation}
which holds for any $F\in L^1(\R^2)$ (Siegel's formula holds of course in any dimension). In a similar vein, formulas \eqref{Ec2} and \eqref{Ec1} follow from the following variant of Siegel's formula: for any $F_1,F_2\in L^1(\R^2)$,
\begin{equation} \label{siegel22}
\int_{X} \sum_{\vecm_1\ne \vecm_2\in\Z^2} F_1(\vecm_1 M+\vecxi) F_2(\vecm_2 M+\vecxi) d\mu(M,\vecxi)  = \int_{\R^2} F_1(\vecx)\,d\vecx  \int_{\R^2} F_2(\vecx)\,d\vecx. 
\end{equation}
We prove this fact in Appendix \ref{app:Siegel}.

Properties (d) and (f) follow from calculations similar to \cite{marklof_2000}. 
We write $g=(1,\vecxi')(M,0)\in G'$ with $M$ as in \eqref{iwasawa}; we use the notation $\vecxi'$ to distinguish this vector from the fixed vector $\vecxi$ that determines the distribution $E_{c,\vecxi}(\veck,I)$.
We have
\begin{equation}\label{eq:main_asymptotics_equation}
\scrN(g,\fC_c(I)) = \sum_{\vecm\in\Z^2} \chi_\phi( (m_1+\xi_1') v^{1/2}  , [(m_1+\xi_1') u + (m_2+\xi_2')] v^{-1/2} ) ,
\end{equation}
where $\chi_\phi$ is the characteristic function of the set $\fC_c(I) k(\phi)^{-1}$. Assume without loss of generality that $\xi_1'\in[-\frac12,\frac12]$. Then, for $v$ sufficiently large,
\begin{equation}\label{eq:main_asymptotics_equation2}
\begin{split}
\scrN(g,\fC_c(I)) & = \sum_{m_2\in\Z} \chi_\phi( \xi_1' v^{1/2}  , [\xi_1' u + (m_2+\xi_2')] v^{-1/2} ) \\
& = v^{1/2} \tilde\chi_\phi( \xi_1' v^{1/2} ) + O(1),
\end{split}
\end{equation}
where
\begin{equation}
\tilde\chi_\phi(\eta) = \int_\R \chi_\phi( \eta  , t ) \, dt .
\end{equation}

Consider first the case of $E_{c,\vecxi}$ for $\vecxi\in\Z^2$, as in \eqref{eq:limit_measures}. Then, for $k_0\to\infty$,
\begin{equation}
\begin{split}
 \sum_{k=k_0}^\infty E_{c,\vecxi}(k,I)
 & = \mu_1\big( \scrN(g,\fC_c(I)) \ge k_0 \big) \\
 &= \mu_1\big( v^{1/2} \tilde\chi_\phi( 0 ) \ge k_0 +O(1) \big) 
 \\
 &= \frac{3}{\pi^2} \int_{u=0}^1 \int_{\phi=0}^{2\pi} \int_{v\ge (k_0+O(1))^2 }  \tilde\chi_\phi(0)^2 \frac{dv\,d\phi\, du}{v^2}\\
 &\sim \frac{3}{\pi^2}\, k_0^{-2} \int_{\phi=0}^{2\pi} \tilde\chi_\phi(0)^2d\phi .
\end{split}
\end{equation}
This proves property (d) for $\vecxi\in\Z^2$. The case of other $\vecxi\in\Q^2$ is analogous. In the case $\vecxi\notin\Q^2$ we use the measure $\mu$ from \eqref{eq:measures} to get 
\begin{equation}
\begin{split}
 \sum_{k=k_0}^\infty E_{c,\vecxi}(k,I)
 & = \mu\big( \scrN(g,\fC_c(I)) \ge k_0 \big) \\
 &= \mu\big( v^{1/2} \tilde\chi_\phi( \xi_1' v^{1/2} )  \ge k_0 +O(1) \big) 
 \\
 &= \frac{3}{\pi^2} \int_{u=0}^1 \int_{\phi=0}^{2\pi} \int_{v\ge (k_0+O(1))^2 } \int_{\xi_1'=-\frac12}^{\frac12} \tilde\chi_\phi( \xi_1' v^{1/2} )^2 \, d\xi_1'\,\frac{dv\,d\phi\, du}{v^2}\\
 &\sim \frac{2}{\pi^2}\, k_0^{-3} \int_{\phi=0}^{2\pi}  \int_{\eta\in\R}  \tilde\chi_\phi(\eta)^3d\phi\, d\eta ,
\end{split}
\end{equation}
which proves property (f).

To deduce \eqref{Ec2}, note that 
\beq \sum_{\veck\in\Z_{\ge 0}^2} k_1 k_2 E_{c}(\veck,I_1\times I_2) = \int_X \sum_{\vecm_1, \vecm_2 \in \Z^2} \chi_{\fC_c(I_1)}(\vecm_1 M + \vecxi) \chi_{\fC_c(I_2)}(\vecm_2 M + \vecxi)  d\mu(M,\vecxi). \eeq
To evaluate the off-diagonal part of the right-hand side we apply \eqref{siegel22}, which yields $|I_1| \, |I_2|$ since, for a bounded interval $I \subset \R$, the area of $\fC_c(I)$ is precisely $|I|$.
As for the diagonal part, let $\fC=\fC_c(I_1) \cap \fC_c(I_2)$ and note that
\begin{multline} \label{Ccdiag} 
\int_X  \sum_{\vecm \in \Z^2} \chi_{\fC}(\vecm M + \vecxi) d\mu(M,\vecxi) = \int_{X_1} \sum_{\vecm \in \Z^2} \int_{\Z^2\backslash\R^2} \chi_{\fC}((\vecm+\bm\zeta) M) \, d\bm\zeta\, d\mu_1(M) \\
= \int_{X_1} \int_{\R^2} \chi_{\fC}(\vecx M) \, d\vecx\, d\mu_1(M) 
=\int_{\R^2} \chi_{\fC}(\vecx) d\vecx = |I_1 \cap I_2| .  
\end{multline}
This concludes the proof of \eqref{Ec2}.

\section{Escape of mass}

We define the abelian subgroups
\[\Gamma_\infty= \left\{\begin{pmatrix}1 &m\\ 0 &1\end{pmatrix}\colon m\in\Z\right\}\subset\Gamma  \]
and
\[\Gamma_\infty'= \left\{\left(\begin{pmatrix}1 &m_1\\ 0 &1\end{pmatrix}, (0,m_2) \right)\colon (m_1,m_2)\in\Z^2\right\}\subset\Gamma' . \]  
These subgroups are the stabilizers of the cusp at $\infty$ of $\Gamma\quot G$ and  $\Gamma'\quot G'$, respectively.

Denote by $\chi_R$ the characteristic function of $[R,\infty)$ for some $R\ge 1$, i.e.~$\chi_R(v)=0$ if $v<R$ and $\chi_R(v)=1$ if $v\ge R$. 
For a fixed real number $\beta$ and a continuous function $f:\R\to\R$ of rapid decay at $\pm\infty$, define the function $F_{R,\beta}\colon \H\times\R^2\to \R$ by 
\begin{equation}
\begin{split}
F_{R,\beta}\left(\tau; \vecxi\right)&=\sum_{\gamma\in\Gamma_\infty\quot \Gamma}\sum_{m\in\Z}f (((\vecxi\gamma^{-1})_1+m)v^{1/2}_\gamma) v^\beta_\gamma\chi_R(v_\gamma) \\
&=\sum_{\gamma\in\Gamma_\infty'\quot \Gamma'} f_\beta(\gamma g) ,
\end{split}
\end{equation}
where $f_\beta:G'\to\R$ is defined by 
\begin{equation}
f_\beta((1,\vecxi) n(u)a(v)k(\varphi)) := f(\xi_1 v^{1/2}) v^\beta \chi_R(v).
\end{equation}
We view $F_{R,\beta}\left(\tau; \vecxi\right)=F_{R,\beta}\left(g\right)$ as a function on $\Gamma'\quot G'$ via the identification \eqref{ida}. 

The main idea behind the definition of $F_{R,\beta}\left(\tau; \vecxi\right)$ is that we have for $v\ge 1$
\begin{equation}
F_{R,\beta}\left(\tau; \vecxi\right) = \sum_{m\in\Z} [f((\xi_1+m)v^{1/2})+f((-\xi_1+m)v^{1/2})] v^\beta \chi_R(v) ,
\end{equation}
which shows that, for the appropriate choice of $f$ and $\beta=\frac12(\sigma_1+\ldots+\sigma_m)$, and $v\ge R$ with $R$ sufficiently large,
\begin{equation}\label{404}
\scrN(g,\fC_1)^{\sigma_1}\cdots \scrN(g,\fC_m)^{\sigma_m} \le F_{R,\beta}\left(\tau; \vecxi\right).
\end{equation}

The following proposition establishes under which conditions there is no escape of mass in the equidistribution of horocycles. It generalizes results in \cite{marklof_pair_correlation_II_2002, marklof_pair_correlation_2003, marklof_mean_square_2005}.

\begin{prop}\label{prop:ilya}
Let $\vecxi\in\R^2$, $\beta\ge 0$, $M\in G$, and $h\in C_0(\R)$. Assume that one of the following hypotheses holds:
\begin{enumerate}[{\rm ({B}1)}]
\item  $\beta<1$.
\item $\vecxi$ is Diophantine of type $\kappa$, and $\beta<1+\frac1{\kappa}$. 
\end{enumerate}
Then
\beq 
\lim_{R\to \infty}\limsup_{v\to 0} \bigg| \int_{u\in\R} F_{R,\beta}\left((1,\vecxi)Mn(u)a(v)\right) h(u)du \bigg|=0.
\eeq
\end{prop}

The proof is organized as follows.
\begin{enumerate}
\item Lemma \ref{lem:daniel};
\item Proof under (B2),  $M=1$;
\item Proof under (B2), $M$ arbitrary;
\item Proof under (B1).
\end{enumerate}

Throughout the proof of Proposition \ref{prop:ilya} we assume without loss of generality that $f$ and $h$ are nonnegative,  $\forall r \ge 1, \forall x \in \R, \, f(rx) \le f(x),$ and that $f$ is even. 

We will need the following

\begin{lemma}\label{lem:daniel} Let $(x,y) \in \mathbb{R}^2$ be Diophantine of type $\kappa$.
Let $f \colon \mathbb{R} \to \mathbb{R}_{\ge 0}$ be continuous and rapidly decreasing.
For every $A > 1, D > 0, T > 1$ and $0 < \epsilon < \frac 1{\kappa}$ we have 
\beq \sum_{D \le c \le 2D} \sum_{1 \le d \le D} \sum_{m \in \mathbb{Z}} f(T(c x + d y + m)) \ll \begin{cases} T^{-A} & \text{ if $D \le T^\epsilon$} \\ 1 & \text{ if $T^\epsilon \le D \le T^{\frac 1{\kappa}}$} \\ \left( \frac D{T^{1/\kappa}} \right)^2 & \text{ otherwise.}  \end{cases} \eeq
\end{lemma}

\begin{proof}
For every $D \le c \le 2D$, every $1 \le d \le D$, and every $m \in \mathbb{Z}$,  we have
\beq|c x+d y+m| \ge \frac C{(c+d)^{\kappa}} \gg \frac C{D^{\kappa}}.\eeq
Combined with the rapid decay of $f$ (in particular, $\forall B > 1, f(t) \ll \frac 1{t^B}$) this gives
\beq\sum_{D \le c \le 2D} \sum_{1 \le d \le D} \sum_{m \in \mathbb{Z}} f(T(c x + d y + m)) \ll D^2 \left( \frac {D^\kappa}T \right)^B\eeq
for every $B > 1$.  When $D \le T^\epsilon$, we obtain $D^2 \left( \frac {D^\kappa}T \right)^B \ll T^{-A}$ for every $A > 1$, hence the first bound.

Now let us divide the sums over $c$ and $d$ into blocks  
\beq\sum_{0 \le c \le T^{1/\kappa}} \sum_{0 \le d \le T^{1/\kappa}} \sum_{m \in \mathbb{Z}} f(T((b+c) x + (b'+d) y + m)).\eeq
The number of such blocks is $\ll \left(\frac D{T^{1/\kappa}}+1 \right)^2.$ 
The distance between any two points from the same block is  $|c' x + d' y + m'|$ for some $(c',d')\neq(0,0)$ with $|c'| \le T^{1/\kappa}$, $|d'| \le T^{1/\kappa}$ and $m' \in \mathbb{Z}$. By our assumption on $(x, y)$ this distance is at least $\frac C{(|c'|+|d'|)^{\kappa}} \ge \frac {C'} T.$ 
Thus, every interval of $\mathbb{R}$ of length $\frac 1T$ contains at most $\ll 1 + \frac 1{C'}$ points from a given block. 
The rapid decay of $f$ (or indeed the fact that  $f(t) \ll \frac 1{t^2}$) gives 
\beq\sum_{0 \le c \le T^{1/\kappa}} \sum_{0 \le d \le T^{1/\kappa}} \sum_{m \in \mathbb{Z}} f(T((b+c) x + (b'+d) y + m)) \ll 1\eeq
for each block. Estimating the sum over the blocks trivially yields the remaining bounds in the statement. 
\end{proof}

\begin{proof}[Proof of Proposition \ref{prop:ilya} under (B2), $M=1$.]
Writing $\vecxi=(y,-x)$ we have 
\begin{align}\label{eq:firstterm}
F_{R,\beta}(\tau; \vecxi)&=
2\sum_{m\in\Z}f\left((m+x) \frac{v^{1/2}}{|\tau|}\right)\frac{v^\beta}{|\tau|^{2\beta}}\chi_R\left(\frac{v}{|\tau|^2}\right)+\\
&+2\sum_{\substack{{(c,d)\in\Z^2}\\{\gcd(c,d)=1}\\{c>0,d\ne 0}}}\sum_{m\in\Z}f\left((cx+dy+m)\frac{v^{1/2}}{|c\tau+d|}\right)
\frac{v^{\beta}}{|c\tau+d|^{2\beta}}\chi_R\left(\frac v{|c\tau+d|^2}\right)\label{eq:secondterm}
\end{align}
for $v<R$.
The integral of the first term tends to zero as $v\to0$: after the change of variables $u=vt$ we get 
\begin{align*}
\int_\R \eqref{eq:firstterm} h(u)\, du &\ll v\int_{t\in\R} \sum_m f\left(\frac{m+x}{v^{1/2}(t^2+1)^{1/2}}\right)\left(\frac1{v(t^2+1)}\right)^\beta \chi_R\left(\frac1{v(t^2+1)}\right)h(vt)dt\\
&\ll v\int_{(vR)^{-1/2}>|t|}\sum_m \left(\frac{m+x}{v^{1/2}(t^2+1)^{1/2}}\right)^{-2B}\left(\frac1{v(t^2+1)}\right)^\beta dt\\
&\ll v^{B-\beta+1}\int_{(vR)^{-1/2}>|t|}(t^2+1)^{B-\beta}dt\\
&\ll v^{1/2} R^{\beta-B-1/2}\to 0
\end{align*}
as $v\to 0$ for every $R$. Here we used that $f(z)\ll z^{-2B}$ for some $B>\frac12$ and the fact that $x$ is not an integer because of the Diophantine condition. 

It remains to analyze the contribution of \eqref{eq:secondterm}. After performing the substitution $t=\left(u+d/c\right)v^{-1}$ we get 
\beq\int_{t=-\infty}^\infty f\left((cx+dy+m)\frac1{\sqrt{c^2 v (t^2+1)}}\right)\frac v{(c^2v(t^2+1))^\beta}\chi_R\left(\frac1{c^2v(t^2+1)}\right) h(vt -d/c)dt.\eeq
From restrictions coming from $\chi_R$ and $h$ we get that $|d|\ll c$, and the implied constant depends only on the support of $h$. 

We need to bound 
\beq\sum_{c=1}^\infty\int_{t\in\R}\left[\sum_{0<|d|\ll c}\sum_{m\in\Z} f\left((cx+dy+m)\frac1{\sqrt{c^2 v (t^2+1)}}\right) \right]\frac  v{(c^2v(t^2+1))^\beta}\chi_R\left(\frac1{c^2v(t^2+1)}\right)\! dt.\label{eq:tobound}\eeq

Now we decompose the region $\dfrac 1{\sqrt{c^2v(t^2+1)}} \ge \sqrt{R}$ into dyadic regions \[2^j \le \dfrac 1{\sqrt{c^2 v (t^2+1)}} < 2^{j+1}\] for $j \gg \log R$.
We can thus bound expression \eqref{eq:tobound} by 
\begin{multline*}\sum_{j \gg \log R} \sum_{c \ge 1} \int_\R \left[\sum_{0<|d|\ll c}\sum_{m\in\Z} f\left((cx+dy+m)\frac1{\sqrt{c^2 v (t^2+1)}}\right) \right]\times \\\times \frac  v{(c^2v(t^2+1))^\beta}\chi_{[2^j, 2^{j+1})} \left(\frac1{\sqrt{c^2v(t^2+1)}}\right)\! dt \end{multline*}
\begin{align} 
\notag &\le\! \sum_{j \gg \log R} \sum_{c \ge 1} \int_\R \left(\sum_{0<|d|\ll c}\sum_{m\in\Z} f(2^j(cx+dy+m)) \right)\!\frac  v{(c^2v(t^2+1))^\beta}\chi_{[2^j, 2^{j+1})} \!\left(\frac1{\sqrt{c^2v(t^2+1)}}\right)\! dt \\
\label{eq:doublesum}&\ll v \sum_{j \gg \log R} 2^{2 \beta j} \int_\R \sum_{0<|d| \ll \frac {2^{-j}}{\sqrt{v(t^2+1)}}} \left(\sum_{\frac {2^{-(j+1)}}{\sqrt{v(t^2+1)}} \le c \le \frac {2^{-j}}{\sqrt{v(t^2+1)}}} \sum_{m\in\Z} f(2^j(cx+dy+m)) \right) dt \end{align}
To bound this we apply Lemma \ref{lem:daniel} with $D \sim \frac {2^{-(j+1)}}{\sqrt{v(t^2+1)}}$ and $T = 2^j$. 
Note that the domain of integration is always restricted to $t^2+1 \le 1/v$ since $2^j c \ge 1$.

In the first range we have the bound
\beq
 v\sum_{j\ge 0}2^{2\beta j}\int_{t^2+1\le 1/v} 2^{-jA}dt\ll v^{1/2}\to 0.
\eeq
For the second range we restrict the domain of integration to $D\ge T^\eps$. For $\rho\in (0,2)$ we have the bound 
\begin{align}\label{four seventeen}
\notag v\sum_{j\ge 0} 2^{2\beta j}\int_{D\ge T^\eps} dt&\ll v\sum_{j\ge 0}2^{2\beta j}\int_{D\ge T^\eps}\left(\frac D{T^\eps}\right)^{2-\rho}dt\\
 &\ll v\sum_{j\ge0} 2^{2\beta j}\int_{t\in\R}\frac{2^{-(j+1)(2-\rho)}}{2^{(2-\rho)j\eps}(v(t^2+1))^{\frac{2-\rho}{2}}}dt\\
\notag &\ll v^{\rho/2}\sum_{j\ge 0} 2^{j(2\beta-(2-\rho)-(2-\rho)\eps)}.
\end{align}
The sum over $j$ converges whenever $\beta<1+\eps-\frac\rho2(1+\eps)$. It is clear that for every $\beta<1+\frac1\kappa$ we can find $\eps<\frac1\kappa$ and $\rho\in(0,2)$ so that this condition is satisfied. Then the contribution of the second range is $v^{\rho/2}\to0$, as needed. 

Write $\delta=\frac1\kappa$. Then, in the third range we have 
\begin{align}\label{eq:third}
\notag v\sum_{j\gg \log R}2^{2\beta j}\int_{t\in \R}D^2T^{-2\delta} dt &\ll v\sum_{j\gg \log R} 2^{2\beta j}\int_{t\in\R}\left(\frac{2^{-j}}{\sqrt{v(t^2+1)}}\right)^2 2^{-2\delta j}dt\\
 &\ll  \sum_{j\gg \log R} 2^{2j(\beta - 1 -\delta)}\to 0
\end{align}
as $R\to \infty$ since $\beta<1+\delta$.

\end{proof}

\begin{proof}[Proof of Proposition \ref{prop:ilya} under (B2), and $M$ arbitrary]
We need to show that 
\beq\label{eq:fullexpression}\adjustlimits\lim_{R\to \infty}\limsup_{v\to 0} \bigg| \int_{u\in\R} F_{R,\beta}\left((1,\vecxi)Mn(u)a(v)\right) h(u)du \bigg|=0\eeq
for a fixed $M\in G$.
Let $M=\left(\begin{smallmatrix}a&b\\c&d\end{smallmatrix}\right)$. Then we can find $\tilde u$, $\tilde v$ and $\theta$ so that 
\[\begin{pmatrix}a&b\\c&d\end{pmatrix}\begin{pmatrix}1&u\\&1\end{pmatrix}\begin{pmatrix}v^{1/2}\\&v^{-1/2}\end{pmatrix}=\begin{pmatrix}1&\tilde u\\&1\end{pmatrix} 
\begin{pmatrix}\tilde v^{1/2}\\&\tilde v^{-1/2}\end{pmatrix}k(\theta).\]
It is well known that 
\beq\label{eq:decomposition}\tilde u+\mathrm{i}\tilde v=\frac{a(u+\mathrm{i}v)+b}{c(u+\mathrm{i}v )+d}=\frac{(au+b)(cu+d)+acv^2}{(cu+d)^2+(cv)^2}+\mathrm{i}\frac{v}{(cu+d)^2+(cv)^2}.\eeq

\textbf{Case A.} If $cu+d$ never vanishes (or equivalently if $-d/c$ is not in the support of $h$), we bound the integral in the statement by a change of variable to $\tilde u$. The Jacobian 
\[j(\tilde u)=\left|\frac{du}{d\tilde u}\right|\]
is bounded away from zero and infinity when $v$ is small enough. Therefore the original integral is equal to 
\beq \label{eq:integral}
\int_{\tilde u\in\R} F_{R,\beta}\left((1,\vecxi)n(\tilde u)a(\tilde v)\right) h(u)j(\tilde u)d\tilde u. 
\eeq
Since $h$ has compact support, let $\supp h\subset [-H,H]$ for some $H>0$. Then, $|\tilde u|$ is at most 
\[\frac{(|a|H+|b|)(|c|H+|d|)+|ac|}{(|c|H-|d|)^2}<\infty\]
for $v\le 1$, and  hence we can find a nonnegative $\tilde h\in C_0(\R)$ such that $\tilde h(\tilde u)\ge h(u) j(\tilde u)$. The integral \eqref{eq:integral} is at most 
\beq\label{eq:integral1}\int_{\tilde u\in\R} F_{R,\beta}\left((1,\vecxi)n(\tilde u)a(\tilde v)\right) \tilde h(\tilde u)d\tilde u. \eeq
Now observe that
\[\tilde v\in \left[\frac{v}{(|c|H-|d|)^2+c^2},\frac{v}{(|c|H-|d|)^2}\right]=I(v).\]
Therefore we have 
\begin{align*}\eqref{eq:integral1}
&\le \sup_{u'\in\supp\tilde h}\int_{\tilde u\in\R} F_{R,\beta}\left((1,\vecxi)n(\tilde u)a(\tilde v(u'))\right) \tilde h(\tilde u)d\tilde u.\\
&\le \sup_{v'\in I(v)}\int_{\tilde u\in\R} F_{R,\beta}\left((1,\vecxi)n(\tilde u)a(v')\right) \tilde h(\tilde u)d\tilde u.
\end{align*}
We then apply Proposition \ref{prop:ilya} with $h$, $u$, and $v$ replace by $\tilde h$, $\tilde u$, and $ v'$, respectively.

\textbf{Case B.} So suppose that $-d/c$ is in the support of $h$. Then we ``flip'' the problem as follows. Let 
\[J=\begin{pmatrix}0&1\\-1&0\end{pmatrix}\in\Gamma,\]
and consider $(J,\vecnull)\in\Gamma'$. Since $F_{R,\beta}$ is left-$\Gamma'$-invariant, we have
\beq
F_{R,\beta}\left((1,\vecxi)Mn(u)a(v)\right)=F_{R,\beta}\left((J,\vecnull)(1,\vecxi)Mn(u)a(v)\right)=F_{R,\beta}\left((1,-\vecxi J)JMn(u)a(v)\right).
\eeq
This effectively switches $\xi_1$ and $\xi_2$, so that any Diophantine condition from the assumptions will be preserved. Now 
\[JM=\begin{pmatrix}
      c&d\\
      -a&-b
     \end{pmatrix}.\]
Repeating the decomposition from \eqref{eq:decomposition} with $JM$ in place of $M$ yields the condition that $au+b$ should never vanish (or equivalently that $-b/a$ is not in the support of $h$). If $au+b\ne 0$ for all $u$ in the support of $h$, then we are done since we can use $\tilde u$, $ v'$, and $\tilde h$ as before.

\textbf{Case C.} Suppose that both $-d/c$ and $-b/a$ are in the support of $h$. They must be distinct as $ad-bc=1$; so we write $h=h_1+h_2$ with $h_1\in C_0(\R)$ not supported in a neighborhood of $-b/a$ and $h_2\in C_0(\R)$ not supported in a neighborhood of $-d/c$. Then we apply the above arguments to $h_1$ and $h_2$ separately, and these functions will fall under cases B and A, respectively. 
\end{proof}

\begin{proof}[Proof of Proposition \ref{prop:ilya} under (B1)]
Let $M=1$; the case of general $M$ can be treated as under (B2). Since $f$ is rapidly decaying and $R\ge 1$, we have
\begin{equation}
F_{R,\beta}\left(\tau; \vecxi\right) \ll_f \overline F_{R,\beta}\left(\tau\right)
\end{equation}
where
\begin{equation}
\overline F_{R,\beta}\left(\tau\right)=\sum_{\gamma\in\Gamma_\infty\quot \Gamma} v^\beta_\gamma\chi_R(v_\gamma).
\end{equation}
Thus
\begin{equation}
\int_{u\in\R} F_{R,\beta}\left(u+\mathrm{i}v;\vecxi \right) h(u)du \ll_{f,h}
\int_0^1 \overline F_{R,\beta}\left(u+\mathrm{i}v\right) du .
\end{equation}
The evaluation of the integral on the right hand side is well known from the theory of Eisenstein series.
We have
\begin{equation}
\overline F_{R,\beta}\left(\tau\right)=v^\beta \chi_R(v) +
2 \sum_{c=1}^\infty  \sum_{\substack{d=1 \\ \gcd(c,d)=1}}^{c-1} \sum_{m\in\Z} \frac{v^\beta}{c^{2\beta} |\tau+\frac{d}{c}+m|^{2\beta}} \chi_R\left(\frac{v}{c^2 |\tau+\frac{d}{c}+m|^2}\right).
\end{equation}
This function is evidently periodic in $u=\re\tau$ with period one, and its zeroth Fourier coefficient is
(we denote by $\varphi$ Euler's totient function)
\begin{align}
\notag \int_0^1  \overline F_{R,\beta}\left(u+\mathrm{i}v\right) du   & = 
v^\beta \chi_R(v) +
2 \sum_{c=1}^\infty  \sum_{\substack{d=1 \\ \gcd(c,d)=1}}^{c-1} \frac{1}{c^{2\beta}} \int_\R \frac{v^\beta}{|u+\mathrm{i}v|^{2\beta}} \chi_R\left(\frac{v}{c^2 |u+\mathrm{i}v|^2}\right)  du\\
 & =
v^\beta \chi_R(v) +
2 v^{1-\beta} \sum_{c=1}^\infty  \frac{\varphi(c)}{c^{2\beta}} \int_\R \frac{1}{(t^2+1)^\beta} \chi_R\left(\frac{1}{vc^2 (t^2+1)}\right)  dt.
\end{align}
The first term vanishes for $v<R$, and the second term is bounded from above by 
\begin{equation}
2 v^{1-\beta} \sum_{c=1}^\infty  \frac{1}{c^{2\beta-1}} \int_\R \frac{1}{(t^2+1)^\beta} \chi_R\left(\frac{1}{vc^2 (t^2+1)}\right) \, dt = 2 v^{1/2} \sum_{c=1}^\infty K_R(cv^{1/2})
\end{equation}
with the function $K_R:\R_{>0}\to \R_{\ge 0}$ defined by
\begin{equation}
K_R(x)=  \frac{1}{x^{2\beta-1}} \int_\R \frac{1}{(t^2+1)^\beta} \chi_R\left(\frac{1}{x^2 (t^2+1)}\right) \, dt .
\end{equation}
We have $K_R(x)\ll \max\{ 1, x^{2\beta-1}\} \le \max\{ 1, x^{-1/2}\}$ and furthermore $K_R(x)=0$ if $x>R^{-1/2}$. Thus
\begin{equation}
\lim_{v\to 0} v^{1/2} \sum_{c=1}^\infty K_R(cv^{1/2}) = \int_\R K_R(x) dx,
\end{equation}
which evaluates to a constant times $R^{-(1-\beta)}$. 
\end{proof}

\begin{prop}\label{fortdep} Fix  $A, B>1$, $M\in G$, and let $\fK= [\frac1A,A]\times [-B,B]$.
Under the assumptions of Proposition \ref{prop:ilya}, we have
\beq \label{uniconv}\adjustlimits \lim_{R \to \infty} \limsup_{v \to 0} \sup_{(a,b)\in\fK} 
\bigg|\int_\R F_{R,\beta} \left( (1,\vecxi) Mn(u) \begin{pmatrix} a & 0 \\ 0 & a^{-1} \end{pmatrix} \begin{pmatrix} 1 & 0 \\ b & 1 \end{pmatrix} a(v)\right) h(u) du \bigg|= 0.
\eeq
\end{prop}

\begin{proof}
To use Proposition \ref{prop:ilya} we pass to Iwasawa coordinates, in particular we find $\tilde{u}$ and $\tilde{v}$ such that \beq \begin{pmatrix} 1 & u \\ 0 & 1 \end{pmatrix} \begin{pmatrix} a & 0 \\ 0 & a^{-1} \end{pmatrix} \begin{pmatrix} 1 & 0 \\ b & 1 \end{pmatrix} \begin{pmatrix} v^{1/2} & 0 \\ 0 & v^{-1/2} \end{pmatrix} = \begin{pmatrix} 1 & \tilde{u} \\ 0 & 1 \end{pmatrix} \begin{pmatrix} \tilde{v}^{1/2} & 0 \\ 0 & \tilde{v}^{-1/2} \end{pmatrix} k(\theta) \eeq for some $\theta \in [0, 2 \pi].$
A quick calculation yields 
\[\tilde{u} = u + \frac {ba^2v^2}{1+(bv)^2},\text{ and }\tilde{v} = \frac {a^2v}{1+(bv)^2}.\] 
Thus, the integral in \eqref{uniconv} is  \beq \label{newfunctions} \int_\R {F_{R,\beta}} \left( (1, \vecxi) M\begin{pmatrix} 1 & \tilde{u} \\ 0 & 1 \end{pmatrix} \begin{pmatrix} \tilde{v}^{1/2} & 0 \\ 0 & \tilde{v}^{-1/2} \end{pmatrix} \right) \tilde{h}(\tilde{u}) d \tilde{u} \eeq
where $\tilde{h}$ is nonnegative and compactly supported.
Observe that $\frac{v}{A^2(1+B^2)}\le \tilde{v} \le A^2 v$ as $v$ goes to $0$, hence identifying $n(\tilde u) a(\tilde v)$ with $\tilde{u}+\mathrm{i}\tilde{v}$ allows us to apply Proposition \ref{prop:ilya} with  $h := \tilde{h},$ which ensures that \beq \label{clfornewfns} \adjustlimits\lim_{R \to \infty} \limsup_{\tilde{v} \to 0} |\eqref{newfunctions}| = 0. \eeq This gives the desired result since the left-hand side of \eqref{uniconv} is  that of \eqref{clfornewfns}.
\end{proof}

\begin{prop}\label{spherical77} 
Let $\lambda$ be a Borel probability measure on $\T$ with continuous density, $\vecxi\in\R^2$ and $\beta\ge 0$ so that (B1) or (B2) holds. Then
\beq \label{uniconv2} \lim_{R \to \infty} \limsup_{t \to \infty} \bigg|\int_\T F_{R,\beta} \left( (1,\vecxi) Mk(2\pi\alpha) \Phi^t\right) \lambda(d\alpha) \bigg|= 0.
\eeq
\end{prop}

\begin{proof} This follows from  Proposition \ref{fortdep} by the same argument as in the proof of \cite[Corollary 5.4]{marklof_strombergsson_free_path_length_2010}. 
\end{proof}

\section{The main lemma\label{sec:main}}

As explained in the introduction, the following key lemma establishes that Theorem \ref{th:main0}
follows from Theorem \ref{th:prelim} under the stated assumptions.

\begin{lemma}
Under the assumptions of Theorem \ref{th:main0}, 
\begin{equation}\label{limit:main4}
\lim_{K\to\infty} \limsup_{T\to\infty} \left|\M_\lambda(T,\vecs)-\M_\lambda^{(K)}(T,\vecs)\right| = 0 .
\end{equation}
\end{lemma}

\begin{proof}
We have
\begin{equation}
\left|\M_\lambda(T,\vecs)-\M_\lambda^{(K)}(T,\vecs)\right|
\le \int_{\scrN_{c,T}(\overline I,\alpha)\ge K} (\scrN_{c,T}(\overline I,\alpha)+1)^{\sigma} \lambda(d\alpha)
\end{equation}
where $\overline I=\cup_j I_j$ and $\sigma=\sum_j \re_+(s_j)$. The statement now follows from Proposition \ref{spherical77} after following a chain of inequalities from eq.~\eqref{crude}, Lemma \ref{lem:upper}, and eq.~\eqref{404}.
\end{proof}

This completes the proof of Theorem \ref{th:main0}.

\section{Singular Diophantine conditions\label{sec:singular}}

The crucial step necessary to extend Theorem \ref{th:main0} to the Diophantine condition stated in Remark \ref{rem:div} is the following lemma. 

\begin{lemma} \label{lem:daniel2} 
Let $(x,y)=\vecn \omega + \vecl$ where $\vecn \in \Z^2 \setminus \{\vecnull \}$, $\omega\in\R$ is Diophantine of type $\frac{\kappa}{2}$, $\vecl \in \Q^2$ and $\det(\vecn, \vecl) \notin \Z$.
Let $f \colon \R \to \R_{\ge 0}$ be continuous and rapidly decreasing.
For every $A > 1, D > 0, T > 1$ and $0 < \epsilon < \frac 2{\kappa}$ we have 
\begin{equation}
\sum_{\substack{D \le c \le 2D \\ 1 \le d \le D \\ \gcd(c,d) = 1 }} \sum_{m \in \mathbb{Z}} f(T(c x + d y + m)) \ll \begin{cases} T^{-A} & \text{ if $D \le T^\epsilon$} \\ D & \text{ if $T^\epsilon \le D \le T^{2/\kappa}$} \\ \frac {D^2}{T^{2/\kappa}} & \text{ otherwise.}  \end{cases} 
\end{equation}
The estimates remain valid in the range $D \ge T^\epsilon$ without restricting the sum to $\gcd(c,d) = 1$.
\end{lemma}

\begin{proof} 
Write  $\vecn=\begin{pmatrix} n_1 \\ n_2 \end{pmatrix} $. 
There exists a matrix $\gamma \in \SL_2(\Z)$ such that $\begin{pmatrix} n_1 \\ n_2 \end{pmatrix} = \gamma \begin{pmatrix} 0 \\ n \end{pmatrix}$ where $n = \gcd(n_1, n_2)$.
Let $\mathcal{B} = [1, 2] \times (0, 1]$.
We have
\begin{equation} 
\sum_{\substack{(c, d) \in D \mathcal{B} \cap \Z^2 \\ \gcd(c,d) = 1 }} \sum_{m \in \Z} f(T(c x + d y + m)) \\
= \sum_{\substack{(c,d) \in D \tilde{\mathcal{B}} \cap \Z^2 \\ \gcd(c,d) = 1 }} \sum_{m \in \Z} f(T(cs_1 + d (n \omega + s_2) + m))
\end{equation}
where $\tilde{\mathcal{B}} =\mathcal{B}  \gamma^{-1}$ and $\begin{pmatrix} s_1 \\ s_2 \end{pmatrix} = \gamma^{-1} \vecl\in\Q^2$.
Note that $\det(\vecn, \vecl) = \det(\gamma^{-1} (\vecn, \vecl)) = -ns_1$, which implies $s_1 \notin \Z$.

For the second and third ranges we estimate, proceeding as in \cite[Lemma 6.6]{marklof_pair_correlation_2003}, 
\begin{equation}
\sum_{|d|\ll D} \sum_{m \in \Z} f(T(cs_1 + d (n \omega + s_2) + m)) \ll 1 + \frac {D}{T^{2/\kappa}} 
\end{equation} 
which holds uniformly in $c \in \Z$. The sum over $c$ is bounded trivially by a constant times $D$, and we obtain the desired result in the second and third range.
For the first range, if $d=0$ then $c = \pm 1$ and the fact that $s_1 \in \Q \setminus \Z$ yields the desired bound for this term. For the remaining sum over $d \ne 0$, we apply the same argument as in \cite[Lemma 6.6]{marklof_pair_correlation_2003}.
\end{proof}

The proof of Proposition \ref{prop:ilya} can now be adapted to hold subject to 
\begin{enumerate}[{\rm ({B}3)}]
\item  $\vecxi=\vecn \omega + \vecl$ where $\vecn \in \Z^2 \setminus \{\vecnull \}$, $\omega\in\R$ is Diophantine of type $\frac{\kappa}{2}$, $\vecl \in \Q^2$, $\det(\vecn, \vecl) \notin \Z$, and $\beta<1+\frac1{\kappa}$.
\end{enumerate}

Lemma \ref{lem:daniel2} replaces Lemma \ref{lem:daniel} in the proof. 
The estimates for the first range are obtained as before, keeping the restriction $\gcd(c,d)=1$.

For the second range we restrict the domain of integration to $D\ge T^\eps$. In place of \eqref{four seventeen} we have, for any $\rho\in (0,1)$, 
\begin{align}
\notag v\sum_{j\ge 0} 2^{2\beta j}\int_{D\ge T^\eps} D\, dt&\ll v\sum_{j\ge 0}2^{2\beta j}\int_{D\ge T^\eps} D\left(\frac D{T^\eps}\right)^{1-\rho}dt\\
 &\ll v\sum_{j\ge0} 2^{2\beta j}\int_{t\in\R}\frac{2^{-(j+1)(2-\rho)}}{2^{(1-\rho)j\eps}(v(t^2+1))^{\frac{2-\rho}{2}}}dt\\
\notag &\ll v^{\rho/2}\sum_{j\ge 0} 2^{j(2\beta-(2-\rho)-(1-\rho)\eps)}.
\end{align}
The sum over $j$ converges whenever $\beta<1+\frac12 \eps-\frac\rho2(1+\eps)$. Now, for every $\beta<1+\frac1{\kappa}$ we can find $\eps<\frac2\kappa$ and $\rho\in(0,1)$ so that this condition is satisfied. Then the contribution of the second range is $v^{\rho/2}\to0$, as needed. 
In the third range, eq.~\eqref{eq:third} remains unchanged (use again $\delta=\frac 1{\kappa}$).

The remaining sections of the proof of Proposition \ref{prop:ilya} do not require any amendments. Note that (B3) is invariant under $\vecxi\mapsto\vecxi \gamma$ for any $\gamma\in\SL(2,\Z)$. This implies that Propositions \ref{fortdep} and \ref{spherical77} hold subject to (B3), with the same proofs as for (B1), (B2).

\begin{appendix}
\section{Second mixed moment vs.\ two-point correlations}\label{324}

We will show in this section that Corollary \ref{Cor1} implies Corollary \ref{cor2}. The proof is in fact independent of the specific choice of the sequence of $\alpha_j$ as long as they satisfy the conclusion of Corollary \ref{Cor1}. The reverse implication ``Corollary \ref{cor2} $\Rightarrow$ Corollary \ref{Cor1}'' follows from a similar, even simpler, argument.

Assume throught this section that the statement of Corollary \ref{Cor1} holds.



\begin{lemma}\label{lem:nodiagonal}
Let $h\in C(\T)$ and $I_1$ and $I_2$ be bounded intervals in \R. Then 
\begin{multline}\label{eq:nondiag}
\lim_{T\to\infty} \int_{\alpha\in \T}\sum_{\substack{{1\le j_1,j_2\le N} \\ {j_1\ne j_2}\\{m_1,m_2\in\Z}}} \chi_{I_1} (N (\alpha_{j_1}-\alpha+m_1))\,\chi_{I_2}(N (\alpha_{j_2}-\alpha+m_2)) h(\alpha)\,d\alpha \\
=|I_1| |I_2|\int_{\alpha\in\T}h(\alpha)\,d\alpha.
\end{multline}
\end{lemma}

\begin{proof}
By Corollary \ref{Cor1}, the left hand side of \eqref{eq:nondiag} {\em without the restriction $j_1\neq j_2$} converges to
\beq|I_1\cap I_2|\int_{\alpha\in \T}h(\alpha)d\alpha+|I_1||I_2|\int_{\alpha\in \T}h(\alpha)d\alpha\label{eq:bothwithh}.\eeq
To identify the contribution of the diagonal $j_1= j_2$, note that for $N$ sufficiently large,
\begin{multline}\label{new2}
 \sum_{\substack{1\le j\le N\\ m_1, m_2\in\Z}}\int_{\alpha\in\T} \chi_{I_1} (N (\alpha_j - \alpha +m_1)) \, \chi_{I_2}(N(\alpha_j-\alpha+m_2)) h(\alpha) \,d\alpha \\ = \sum_{\substack{1\le j\le N\\m\in \Z}}\int_{\alpha\in\T} \chi_{I_1\cap I_2} (N (\alpha_j - \alpha +m)) h(\alpha) \,d\alpha .
\end{multline}
Because $h$ is continuous and $|\alpha-\alpha_{j}|\ll_{I_1,I_2} 1/N$, for any given $\eps>0$ there is $N_0$ such that for all $N\ge N_0$ we have $|h(\alpha)-h(\alpha_{j})|< \eps$ for all $\alpha\in\T$, $j\le N$. Therefore
\begin{multline}\label{eq:comparison}\bigg| \sum_{\substack{1\le j\le N\\m\in \Z}}\int_{\alpha\in\T} \chi_{I_1\cap I_2} (N (\alpha_j - \alpha +m)) h(\alpha) d\alpha  
-  \sum_{\substack{1\le j\le N\\m\in \Z}}\int_{\alpha\in\T} \chi_{I_1\cap I_2} (N (\alpha_j - \alpha +m)) h(\alpha_j) d\alpha \bigg|
\\ <  \eps \sum_{\substack{1\le j\le N\\m\in \Z}}\int_{\alpha\in\T} \chi_{I_1\cap I_2} (N (\alpha_j - \alpha +m))  d\alpha 
 =\eps |I_1\cap I_2| .
\end{multline} 
The right hand side of \eqref{new2} is thus, up to lower order terms,
\begin{equation}
\begin{split}
\sum_{\substack{1\le j\le N\\m\in \Z}} h(\alpha_j) \int_{\alpha\in\T} \chi_{I_1\cap I_2} (N (\alpha_j - \alpha +m))  d\alpha & = \frac{| I_1\cap I_2|}{N}\sum_{1\le j\le N} h(\alpha_j) \\
& \to |I_1\cap I_2| \int_{\alpha\in\T}h(\alpha)d\alpha
\end{split}
\end{equation}
as $N\to\infty$, since the $\alpha_j$ are uniformly distributed mod 1. This confirms that the second summand of \eqref{eq:bothwithh} is the off-diagonal contribution appearing in Lemma \ref{lem:nodiagonal}, as needed. 
\end{proof}

\begin{lemma}\label{lem:nodiagonalindex}
Let $g\in C(\T^2)$ and $I_1$ and $I_2$ be bounded intervals in $\R$. Then 
\begin{multline}\label{eq:lolo}
\lim_{T\to\infty}  \sum_{\substack{{1\le j_1,j_2\le N} \\ {j_1\ne j_2}\\{m\in\Z}}} g(\alpha_{j_2},\alpha_{j_2}) \int_{t\in\R} \chi_{I_1} (N (\alpha_{j_1}+t))\chi_{I_2}(N (\alpha_{j_2}+t+m))  \, dt  \\
=|I_1| |I_2|\int_{\alpha\in\T} g(\alpha,\alpha) d\alpha.
\end{multline}
\end{lemma}

\begin{proof}
We write the left hand side of \eqref{eq:lolo} as
\begin{equation}
\int_{\alpha\in \T}\sum_{\substack{{1\le j_1,j_2\le N} \\ {j_1\ne j_2}\\{m_1,m_2\in\Z}}} \chi_{I_1} (N (\alpha_{j_1}-\alpha+m_1))\chi_{I_2}(N (\alpha_{j_2}-\alpha+m_2)) g(\alpha_{j_2},\alpha_{j_2}) d\alpha .
\end{equation}
Note that $|\alpha-\alpha_{j_i}|\ll_{I_i} 1/N$ ($i=1,2$). By the continuity of $g$, for any given $\eps>0$ there is $N_0$ such that for all $N\ge N_0$, we have $|g(\alpha,\alpha)-g(\alpha_{j_1},\alpha_{j_2})|< \eps$ for all $\alpha\in\T$ and $j_1,j_2\le N$. We may therefore replace $g(\alpha_{j_1},\alpha_{j_2})$ by $g(\alpha,\alpha)$ with error at most 
\begin{align}
\bigg| \int_{\alpha\in \T} & \sum_{\substack{{1\le j_1,j_2\le N} \\ {j_1\ne j_2}\\{m_1,m_2\in\Z}}} \chi_{I_1} (N (\alpha_{j_1}-\alpha+m_1))\chi_{I_2}(N (\alpha_{j_2}-\alpha+m_2)) g(\alpha,\alpha)d\alpha  \label{eq:hindex} \\ 
& - \int_{\alpha\in \T}\sum_{\substack{{1\le j_1,j_2\le N} \\ {j_1\ne j_2}\\{m_1,m_2\in\Z}}} \chi_{I_1} (N (\alpha_{j_1}-\alpha+m_1))\chi_{I_2}(N (\alpha_{j_2}-\alpha+m_2)) g(\alpha_{j_2},\alpha_{j_2}) d\alpha\bigg| \\ & < \eps \int_{\alpha\in \T}\sum_{\substack{{1\le j_1,j_2\le N} \\ {j_1\ne j_2}\\{m_1,m_2\in\Z}}} \chi_{I_1} (N (\alpha_{j_1}-\alpha+m_1))\chi_{I_2}(N (\alpha_{j_2}-\alpha+m_2)) d\alpha \\ & \ll \eps |I_1||I_2|,
\end{align} 
where the last inequality follows from Lemma \ref{lem:nodiagonal} with the choice $h=1$. 
We conclude the proof by noting that \eqref{eq:hindex} converges to the desired answer: apply Lemma \ref{lem:nodiagonal} with the choice $h(\alpha)=g(\alpha,\alpha)$.
\end{proof}

Corollary \ref{cor2} now follows from Lemma \ref{lem:nodiagonalindex} by approximating $f\in C_0(\T^2\times\R)$ from above and below by finite linear combinations of functions of the form
\begin{equation}
\widetilde f(x,y,z) = g(x,y) \int_{t\in\R}  \chi_{I_1}(z+t) \chi_{I_2}(t) dt,
\end{equation}
for suitable choices of $g\in C(\T^2)$ and bounded intervals $I_1$, $I_2\subset\R$.

\section{A variant of Siegel's formula}\label{app:Siegel}

Eq.\ \eqref{siegel22} is a special case of the following. (As noted in the case of Siegel's formula, all of the following statements remain valid for $\R^4$ replaced by $\R^{2n}$.)

\begin{prop}\label{siegelvariant} If $F \in L^1(\R^4)$, then
\begin{equation} \label{siegel25}
\adjustlimits\int_{X} \sum_{\vecm_1\ne \vecm_2\in\Z^2} F(\vecm_1 M+\vecxi,\vecm_2 M+\vecxi) \, d\mu(M,\vecxi)  = \int_{\R^4} F(\vecx)\,d\vecx. 
\end{equation} 
\end{prop}

\begin{proof} 
The density of $C_0(\R^4)$ in $L^1(\R^4)$ and an application of Lebesgue's monotone convergence theorem allow us to assume that $F \in C_0(\R^4)$.
In addition, we assume that $F$ is non-negative.

For every $M \in \SL(2, \R)$ and every $\vecxi \in \R^2$ we set $\bm \zeta = \vecxi M^{-1}$ and thus rewrite
\begin{align} 
&\adjustlimits\int_X  \sum_{\vecm_1\ne \vecm_2\in\Z^2} F(\vecm_1 M+\vecxi,\vecm_2 M+\vecxi) \, d\mu(M,\vecxi) \\ 
&= \int_{X_1}\adjustlimits \int_{\Z^2 \quot \R^2} \sum_{\vecm_1\ne \vecm_2\in\Z^2} F((\vecm_1 + \bm \zeta)M,(\vecm_2 +\bm \zeta)M) \, d\bm \zeta \, d\mu_1(M). \end{align}
Setting $\bm x = \vecm_1 + \bm\zeta$ and $\vecm = \vecm_2 - \vecm_1$, we get that this is equal to
\beq\label{SB1} 
\int_{X_1} \sum_{\vecm \ne \bm 0} \int_{\R^2} F(\bm x M,(\vecm + \bm x)M) \, d \bm x \, d\mu_1(M) 
\eeq
where the non-negativity of $F$ allows the interchange of integration and summation.
A unimodular ($M \in \SL(2,\R)$) change of variables yields that \eqref{SB1} is equal to
\beq \label{SB2} 
\adjustlimits\int_{X_1} \sum_{\vecm \ne \bm 0} \int_{\R^2} F(\bm x,\vecm M + \bm x) \, d \bm x \, d\mu_1(M). 
\eeq
An application of Siegel's formula turns this into
\beq \int_{\R^2} \int_{\R^2} F(\bm x,\bm y + \bm x) \, d\bm y \, d\bm x =
\int_{\R^4} F(\bm x') \, d\bm x' \eeq
as desired.
\end{proof}

\end{appendix}

\bibliographystyle{plain}
\bibliography{bibliography}

\footnotesize
\parindent=0pt

\textsc{Daniel El-Baz, School of Mathematics, University of Bristol, Bristol BS8~1TW, U.K.} \texttt{daniel.el-baz@bristol.ac.uk}
\smallskip

\textsc{Jens Marklof, School of Mathematics, University of Bristol, Bristol BS8~1TW, U.K.} \texttt{j.marklof@bristol.ac.uk}
\smallskip

\textsc{Ilya Vinogradov, School of Mathematics, University of Bristol, Bristol BS8~1TW, U.K.} \texttt{ilya.vinogradov@bristol.ac.uk}

\end{document}